\documentclass[11pt,a4paper]{article}
\usepackage[T1]{fontenc}
\usepackage{lmodern}
\usepackage[margin=28mm]{geometry}
\usepackage{microtype,amsmath,amssymb,amsthm,mathtools,booktabs,array}
\usepackage{enumitem,fancyvrb}
\usepackage[hidelinks,pdftitle={Crown graphs maximise the representation number of bipartite graphs},pdfauthor={Matthew J. Colbrook and Catherine Drysdale}]{hyperref}
\setlist{itemsep=3pt,topsep=5pt}
\newtheorem{theorem}{Theorem}[section]
\newtheorem{lemma}[theorem]{Lemma}
\newtheorem{proposition}[theorem]{Proposition}
\newtheorem{corollary}[theorem]{Corollary}
\theoremstyle{definition}

\newtheorem{example}[theorem]{Example}
\theoremstyle{remark}

\numberwithin{equation}{section}
\newcommand{\RN}{\operatorname{R}}
\newcommand{\rev}{\operatorname{rev}}
\newcommand{\st}{\ast}
\newcommand{\ceil}[1]{\left\lceil#1\right\rceil}
\title{Crown graphs maximise the representation number\\of bipartite graphs}
\author{Matthew J. Colbrook\thanks{Department of Applied Mathematics and
Theoretical Physics, University of Cambridge, UK.
Email: \href{mailto:m.colbrook@damtp.cam.ac.uk}{\texttt{m.colbrook@damtp.cam.ac.uk}}.}
\and Catherine Drysdale\thanks{School of Mathematical Sciences,
Lancaster University, UK.
Email: \href{mailto:c.drysdale@lancaster.ac.uk}{\texttt{c.drysdale@lancaster.ac.uk}}.}}
\date{}
\begin{document}
\maketitle
\begin{abstract}
The representation number of a graph is the least positive integer $k$ for which its vertices can be arranged in a word, each occurring $k$ times, so that two distinct letters alternate precisely when the corresponding vertices are adjacent. We prove that every bipartite graph on $N\ge9$ vertices has representation number at most $\lceil N/4\rceil$. Together with the known representation number of crown graphs, this settles the conjecture that crowns maximise the representation number among bipartite graphs of the same order. The proof develops a construction of Mozhui and Krishna by reducing the choice of a representing word to an ordering problem for neighbourhoods. We characterise the obstructions to this ordering and use probability estimates to exclude them for all sufficiently large balanced bipartitions. Two finite assertions complete the argument, each established by a checked Boolean unsatisfiability certificate. A refinement of the ordering argument treats the remaining odd part sizes directly. The main theorem and the crown-extremality corollary, including the finite certificate arguments, have also been formalised and verified in Lean~4.
\end{abstract}

\begingroup\small
\par
\endgroup

\section{Introduction}
A graph is \emph{word-representable} if there is a word on its vertex set in which two distinct letters alternate precisely when the corresponding vertices are adjacent. For example, the restriction to two adjacent vertices $a,b$ may be $ababab$, whereas $aa$ or $bb$ in the restriction excludes adjacency. A representing word is \emph{$k$-uniform} if every vertex occurs exactly $k$ times, in which case the graph is \emph{$k$-representable}. The \emph{representation number} $\RN(G)$ is the least such positive integer $k$.

Bipartite graphs admit uniform representations, but finding the least number of occurrences of each vertex (the multiplicity) is a different problem. A natural candidate for the largest representation number is the \emph{crown graph} $H_{n,n}$, obtained from $K_{n,n}$ by deleting a perfect matching. Its many edges impose alternation for almost every pair of vertices in opposite parts, while each deleted edge must still be distinguished.

Glen, Kitaev and Pyatkin~\cite{GKP} proved that
\begin{equation}\label{eq:crown}
                  \RN(H_{n,n})=\ceil{n/2},\qquad n\ge5.
\end{equation}
They proposed the corresponding extremal conjecture~\cite[Conjecture~1]{GKP}. Akg\"un, Gent, Kitaev and Zantema~\cite[Conjecture~5]{AKGZ} stated it explicitly: $H_{n,n}$ has the largest representation number among bipartite graphs on $2n$ vertices. We prove the following bound.

\begin{theorem}\label{thm:main}
Let $G$ be a finite simple bipartite graph on $N\ge9$ vertices. Then
\begin{equation}\label{eq:main}
                         \RN(G)\le\ceil{N/4}.
\end{equation}
\end{theorem}

The theorem includes disconnected graphs, isolated vertices, and vertices with identical neighbourhoods. The two parts may have different sizes. Combining it with~\eqref{eq:crown} gives the extremal statement directly.

\begin{corollary}\label{cor:crown}
For every $n\ge1$,
\[
 \max\{\RN(G):G\text{ is bipartite and }|V(G)|=2n\}
       =\RN(H_{n,n}).
\]
The common value is $2$ for $1\le n\le3$, $3$ for $n=4$, and $\lceil n/2\rceil$ for $n\ge5$.
\end{corollary}
\begin{proof}
For $n\ge5$, Theorem~\ref{thm:main} gives the upper bound, and $H_{n,n}$ attains it by~\eqref{eq:crown}. For the small crown values, see~\cite[Section~1.4]{AKGZ}. For $n=4$, every bipartite graph has a part of size at most four, so Lemma~\ref{lem:baseline}, with isolated vertices added if necessary, gives the upper bound $3$.

For $n\le3$, the connected-graph classification in~\cite[Section~2, Table~1]{AKGZ} gives representation number at most $2$ for every connected bipartite graph on at most six vertices. The only connected exceptions among graphs of that order are the triangular prism and the wheel with a five-vertex rim, both nonbipartite. The bound extends to disconnected graphs by concatenating $2$-uniform representations of their components; letters in different components then fail to alternate. Complete components have $2$-uniform representations obtained by repeating a permutation.
\end{proof}

Corollary~\ref{cor:crown} answers affirmatively the question, ``Is it true that out of all bipartite graphs, crown graphs require longest word-representants?'', with word length understood as the minimum length of a uniform representation. Among bipartite graphs on a fixed even number $N\ge10$ of vertices, this maximum length is $N\lceil N/4\rceil$.

Theorem~\ref{thm:main} and Corollary~\ref{cor:crown} have been formalised in Lean~4 and verified by its kernel. The formal proofs include the obstruction arguments, probability bounds, both finite certificates and the reductions to arbitrary bipartition sizes. The Lean sources, certificate data and verification programs are available in the accompanying repository~\cite{CDrepo}. Appendix~\ref{app:verification} describes its contents and gives reproduction instructions.

\subsection{Previous work and the remaining difficulty}
Kitaev and Lozin~\cite{KL} give a systematic account of word-representable graphs. Related bounds arise by requiring a representing word to be a concatenation of permutations of the whole vertex set. The least number of such permutations is the \emph{permutation representation number}. Mozhui and Krishna~\cite{MKperm} bound this more restrictive invariant using neighbourhood inclusion and its relation to the dimension of partially ordered sets.

Mozhui and Krishna~\cite{MK} established the bound outside the case of two parts of equal even size, and treated several subclasses of that remaining case. Their construction pairs the vertices of one part and builds a word locally around each pair. One additional permutation then prevents vertices in the other part from alternating. This gives a bound one larger than the required value in the unresolved case.

Complementary results of Hefty, Horn, Muir and Owens~\cite{HHMO24} include lower bounds for dense bipartite graphs, an alternative proof of the crown lower bound, and upper bounds in terms of maximum degree. Their subsequent work~\cite{HHMO26} studies lower bounds for random balanced bipartite graphs and hypercubes.

We develop the local construction of Mozhui and Krishna by choosing the orders within each local word. These orders must prevent alternation within the other part without an additional permutation. A neighbourhood assigns to each pair one of three prescribed positions, or a choice between two extreme positions. An ordering lemma shows that these choices succeed precisely when no two completed vectors are strictly ordered in every coordinate. The resulting obstruction problem is much smaller than the original problem of finding a word.

For at least four pairs, every failure is witnessed by either two fixed vectors or three vectors with a single unspecified entry between them. This classification is the main reduction. It explains both why random choices suffice for large graphs and why the remaining finite cases admit short descriptions. Section~\ref{sec:all} completes the proof of Theorem~\ref{thm:main} by treating the other part sizes.

\subsection{Outline of the proof}
Let $A\sqcup B$ be a bipartition. The central case is $|A|=|B|=2k$, where we must construct a $k$-uniform word. The proof proceeds as follows.
\begin{enumerate}[label=(\roman*)]
\item \emph{Replace words by orders.} Pair the vertices of $A$. The local words already handle every pair involving a vertex of $A$. Their restrictions to $B$ are permutations, so the remaining requirement is that each pair in $B$ occurs in both relative orders. Section~\ref{sec:pair} turns this into a condition on vectors of positions.
\item \emph{Identify the obstructions.} Complete each unspecified position by one of two extreme values. A row with two unspecified entries can be dealt with immediately. For the remaining rows, a system of implications gives an exact test. When $k\ge4$, an infeasible system contains a two-row or three-row obstruction (Section~\ref{sec:obstructions}).
\item \emph{Choose a pairing.} For $k\ge8$, any matching of $A$ has a suitable orientation: each pair in $B$ forbids at most two of the $2^k$ orientations. Averaging over matchings improves the estimate for $k=6,7$. For $k=5$, we also use the freedom to choose the extreme values (Section~\ref{sec:prob}).
\item \emph{Resolve the finite cases.} For $k=3,4$, we encode a family obstructing every oriented matching as a Boolean formula and prove that formula unsatisfiable (Section~\ref{sec:finite}). The odd part sizes follow from a refinement of the obstruction argument, including the case of a five-vertex part.
\end{enumerate}

The division into cases reflects the number of available choices. Under the simplest choice of extreme values, a given pair of neighbourhoods obstructs at most two orientations. Thus the number of potential obstructions grows quadratically in $k$, while the number of choices grows exponentially. This settles all large cases without changing the matching. For small $k$, we also vary the matching and the unspecified entries. The obstruction classification determines how these choices enter the probability estimates and the two finite certificates.

\subsection{Conventions}
All graphs are finite and simple. In the constructions, $V(G)=A\sqcup B$ is a fixed bipartition and $N=|V(G)|$. The restriction $w|_S$ is obtained by deleting from $w$ all letters outside $S$; $\rev(w)$ denotes the reversal of $w$, and $a\sim b$ means that $a$ and $b$ are adjacent. Neighbourhoods are allowed to coincide. An \emph{oriented matching} of $A$ is a partition into pairs together with a first and second vertex in each pair; it is unrelated to an orientation of the edges of $G$.

Deleting letters preserves representation of the induced subgraph. We may therefore add isolated vertices to one part and remove them after constructing a word. A cyclic rotation of a uniform word also preserves its represented graph: a two-letter word with equal multiplicities alternates linearly if and only if it alternates cyclically.

\section{From neighbourhoods to orders}\label{sec:pair}
Given two vertices $x,y\in A$, the word $xyxy$ provides five gaps in which a letter $b\in B$ can be inserted. The middle gap makes $b$ alternate with both $x$ and $y$; the adjacent inner gaps make it alternate with exactly one of them; either outer gap makes it alternate with neither. This is the local construction of~\cite[Section~2]{MK}. The freedom in the outer gaps, and in ordering letters assigned to the same gap, will be used to handle the nonedges within $B$.

Assume $|A|=2k$, with $k\ge3$. Choose an oriented perfect matching $(x_1,y_1),\ldots,(x_k,y_k)$ of $A$. For each $b\in B$, its initial rank in coordinate $i\in\{1,\ldots,k\}$ is
\begin{equation}\label{eq:rank}
 \rho_i(b)=
 \begin{cases}
 \st,&b\not\sim x_i,\ b\not\sim y_i,\\
 1,&b\sim x_i,\ b\not\sim y_i,\\
 2,&b\sim x_i,\ b\sim y_i,\\
 3,&b\not\sim x_i,\ b\sim y_i.
 \end{cases}
\end{equation}
The numbers $0,1,2,3,4$ label the five gaps from left to right. The initial vector is $\rho(b)=(\rho_1(b),\ldots,\rho_k(b))$. A star denotes an unspecified entry, which can independently be replaced by $0$ or $4$. A completion is a vector $r(b)=(r_1(b),\ldots,r_k(b))\in\{0,1,2,3,4\}^k$, with $r_i(b)=\rho_i(b)$ at every nonstarred coordinate and $r_i(b)\in\{0,4\}$ at every starred coordinate. Its entries are the completed \emph{ranks}; they specify insertion positions in the local word. For one pair, the possibilities are
\begin{center}
\begin{tabular}{ccc}
\toprule
Neighbours of $b$ in $\{x,y\}$ & Rank & Restriction to $\{x,y,b\}$\\
\midrule
$\varnothing$ & $0$ or $4$ & $bxyxy$ or $xyxyb$\\
$\{x\}$ & $1$ & $xbyxy$\\
$\{x,y\}$ & $2$ & $xybxy$\\
$\{y\}$ & $3$ & $xyxby$\\
\bottomrule
\end{tabular}
\end{center}
In each coordinate, the letters in $B$ must be listed in nondecreasing rank order. We need to choose these orders so that no pair of letters has the same relative order in every coordinate. Strict inequalities in every coordinate would prevent this. The next lemma shows that they are the only obstruction.

\begin{lemma}[Ordering the rows]\label{lem:orders}
Let $k\ge2$. Suppose vectors $r(b)\in\{0,1,2,3,4\}^k$, $b\in B$, have no strictly coordinatewise comparable pair. Thus no distinct $b,c$ satisfy $r_i(b)<r_i(c)$ for every $i$. There are permutations $\pi_1,\ldots,\pi_k$ of $B$, each nondecreasing in its own rank coordinate, such that every pair of letters occurs in both relative orders among the permutations.
\end{lemma}
\begin{proof}
Order the vertices $b\in B$ in $\pi_1$ by increasing $r_1(b)$, breaking ties by decreasing $\sum_{j=2}^k r_j(b)$, and then by any fixed order. For $i>1$, order them by increasing $r_i(b)$, breaking ties in the reverse of $\pi_1$.

Suppose $b$ preceded $c$ in every $\pi_i$. For $i>1$ the ranks cannot tie, because a tie would reverse their $\pi_1$ order. Hence $r_i(b)<r_i(c)$ for all $i>1$. If $r_1(b)<r_1(c)$, the vectors are strictly comparable, contrary to the hypothesis. If their first coordinates tie, the sum of the remaining coordinates is smaller for $b$, so the descending-sum rule places $c$ before $b$ in $\pi_1$, again a contradiction. Interchanging $b,c$ excludes the other constant relative order. Conversely, if $r_i(b)<r_i(c)$ for every $i$, every coordinate-respecting permutation places $b$ before $c$. Thus the absence of such comparisons is also necessary.
\end{proof}

\begin{lemma}[A construction from pairs]\label{lem:blocks}
Let $|A|=2k$ with $k\ge3$. If the stars in \eqref{eq:rank} can be completed so that no two rank vectors are strictly coordinatewise comparable, then $G$ is $k$-representable.
\end{lemma}
\begin{proof}
Take the permutations from Lemma~\ref{lem:orders}. We assemble the local words using the intervening words of~\cite[Section~2]{MK}, and verify all pairs of letters. For $0\le j\le4$, let $B_{i,j}$ be the consecutive subword of $\pi_i$ consisting of the vertices $b$ with $r_i(b)=j$, possibly empty. Form
\begin{equation}\label{eq:block}
 W_i=B_{i,0}\,x_i\,B_{i,1}\,y_i\,B_{i,2}\,x_i\,B_{i,3}\,y_i\,B_{i,4}.
\end{equation}
In the definition of $D_i$, the successor $i+1$ is interpreted modulo $k$, so the successor of $k$ is $1$. Let $D_i$ consist of all $y_j$ with $j\in\{1,\ldots,k\}\setminus\{i,i+1\}$, followed by all corresponding $x_j$, each group in the usual increasing order of the indices $1,\ldots,k$. Define
\begin{equation}\label{eq:word}
                     w=W_1D_1W_2D_2\cdots W_kD_k.
\end{equation}
Each $b\in B$ occurs once in each $W_i$, hence $k$ times. Each member of pair $i$ occurs twice in $W_i$ and once in each of the $k-2$ intervening words that contain its pair, hence also $k$ times.

Vertices from different pairs of $A$ do not alternate: in the block of either one, its two occurrences have no letter from the other pair between them. Members $x_i,y_i$ of the same pair have order $x_i y_i x_i y_i$ in their block and order $y_i x_i$ in an intervening word. Such an intervening word exists because $k\ge3$, and the change of relative order prevents alternation.

Fix $a\in\{x_i,y_i\}$ and $b\in B$. If they are adjacent, \eqref{eq:block} places $b$ between the two copies of $a$ in $W_i$. On the complementary cyclic arc, the other $k-1$ blocks contain one $b$ each, separated by exactly the $k-2$ intervening words containing $a$; the two intervening words next to $W_i$ omit $a$. Thus their entire restriction alternates. If they are nonadjacent, both copies of $a$ lie on the same side of $b$ in their own block, giving consecutive $a$'s in the restricted word.

Finally, the restriction to $B$ is $\pi_1\cdots\pi_k$. Two $B$ letters alternate exactly when they have the same relative order in every permutation. Lemma~\ref{lem:orders} excludes that possibility. Therefore \eqref{eq:word} represents precisely $G$.
\end{proof}

Vertices with identical neighbourhoods may be given identical completed rank vectors, which are not strictly comparable. The tie-breaking rules of Lemma~\ref{lem:orders} ensure that the corresponding letters occur in both relative orders. This observation justifies removing duplicate neighbourhoods in the finite cases and restoring them afterwards.

\section{Completing the unspecified entries}\label{sec:obstructions}
Call an oriented matching \emph{feasible} if its initial vectors $\rho(b)$ admit completions $r(b)$ without a strictly coordinatewise comparable pair. We now give an exact test for this property. The two extreme values make most rows harmless: a completed row containing both $0$ and $4$ can be neither strictly below nor strictly above another row.

Consequently, a row with two or more stars can be made incomparable with every other row by replacing one star by $0$ and another by $4$. Such a row can neither be the smaller nor the larger member of a strict comparison, whatever the other entries are. Discard these rows from the feasibility test.

We index each initial row $\rho(b)$ by its vertex $b\in B$. The remaining rows are either fixed (no stars) or have one star. Associate a Boolean variable $t_b$ to each one-star row: $t_b=0$ assigns rank $0$, and $t_b=1$ assigns rank $4$. For distinct $b,c\in B$, define $b\prec_*c$ to mean that $\rho_i(b)<\rho_i(c)$ at every coordinate where neither row has a star. Then all possible strict comparisons are excluded by the following exact constraints:
\begin{center}
\begin{tabular}{lll}
\toprule
Type of $b,c$ & Condition & Required consequence\\
\midrule
fixed, fixed & $b\prec_*c$ & infeasible\\
one-star, fixed & $b\prec_*c$ & $t_b=1$\\
fixed, one-star & $b\prec_*c$ & $t_c=0$\\
one-star, one-star & $b\prec_*c$ & $t_c=1\Rightarrow t_b=1$\\
\bottomrule
\end{tabular}
\end{center}
For example, a one-star row can be strictly smaller than a fixed row only when its star is assigned $0$; a one-star row can be strictly larger only when its star is assigned $4$.

These constraints give a complete algorithm. To see the last line of the table, a strict comparison between two one-star rows can occur exactly when the smaller row is assigned $0$ and the larger row $4$, provided their nonstarred coordinates have the indicated inequalities. This also holds when their stars occupy the same coordinate. Start with all variables forced to $1$ and close this set under the implications. Feasibility fails exactly when there is a strictly comparable pair of fixed rows or the closure reaches a variable forced to $0$. Otherwise assign $1$ on the closure and $0$ elsewhere. For three pairs, we use these implication paths to describe the finite obstructions in Section~\ref{sec:three-certificate}.

The implication system allows chains of forced choices. For four or more coordinates, every contradictory chain contains a direct conflict: otherwise four increasing positive ranks would have to fit into the three-element set $\{1,2,3\}$. This gives the following classification.

\begin{lemma}[Two-row and three-row obstructions]\label{lem:cores}
For $k\ge4$, feasibility fails if and only if either a fixed pair is strictly comparable, or a one-star row $c$ has its Boolean variable $t_c$ directly forced both to $0$ and to $1$.

In the latter case, if its star is at coordinate $j$, there are fixed rows $L,U$ such that, for every $i\ne j$,
\begin{equation}\label{eq:triple}
                       (\rho_i(L),\rho_i(c),\rho_i(U))=(1,2,3).
\end{equation}
In particular $c$ is adjacent to every vertex of $A$ except the two in pair $j$.
\end{lemma}
\begin{proof}
If a forced-$1$ variable reaches a forced-$0$ variable through implications, take a shortest such path. With the fixed rows that force its endpoints, and reversing the path's order, it gives a sequence
\[
                  L,c_1,\ldots,c_t,U
\]
with $L\prec_*c_1\prec_*\cdots\prec_*c_t\prec_*U$. Here $L$ forces $t_{c_1}=0$, $U$ forces $t_{c_t}=1$, and each consecutive comparison corresponds to an implication in the reverse direction. The $c_i$ have one star each; the fixed rows have none.

If $t=2$, some coordinate avoids both stars, and at that coordinate the four rows would be strictly increasing in $\{1,2,3\}$, impossible. If $t\ge3$, the first four rows $L,c_1,c_2,c_3$ have at most three starred coordinates between them. Since $k\ge4$, a coordinate avoids all those stars, again requiring four strictly increasing values in a three-element set. Thus $t=1$: a single row is directly forced both ways.

At each coordinate other than that row's star, the three values must be strictly increasing in $\{1,2,3\}$, so they are exactly $1,2,3$. This proves \eqref{eq:triple}. The converse is the direct forcing condition already described.
\end{proof}

\begin{example}\label{ex:obstruction}
Consider three vertices $L,c,U\in B$ with initial rows
\[
 \rho(L)=(2,1,1,1),\qquad \rho(c)=(\st,2,2,2),\qquad \rho(U)=(1,3,3,3).
\]
Assigning $0$ to the star gives $r_i(c)<r_i(U)$ for every $i$; assigning $4$ gives $r_i(L)<r_i(c)$ for every $i$. Thus no completion works, although $L$ and $U$ are not strictly comparable. This is why pair obstructions alone are insufficient. The obstruction is entirely local: one row is forced in opposite directions by two fixed rows.
\end{example}

With three coordinates, the stars in the first three intermediate rows may occupy all coordinates, so the proof of Lemma~\ref{lem:cores} no longer excludes a longer implication path. The three-pair certificate therefore accounts for the full implication closure.

For odd part sizes, however, the construction gives an additional restriction. Adjoin an isolated vertex $z$ to the odd part and pair it with a vertex $a$, in the order $(a,z)$. At this coordinate, every initial row has entry $1$ if its vertex is adjacent to $a$, and $\st$ otherwise. The following lemma shows that this restriction excludes longer contradictory paths even when $k=3$. In Section~\ref{sec:all}, a suitable choice of $a$ will also exclude the remaining three-row obstruction.

\begin{lemma}[A distinguished coordinate]\label{lem:special-coordinate}
Suppose $k\ge3$ and, at some coordinate $j$, every initial row has entry $1$ or $\st$. Then feasibility fails if and only if there are fixed rows $L,U$ and a one-star row $c$ such that
\[
 (\rho_j(L),\rho_j(c),\rho_j(U))=(1,\st,1),\qquad
 (\rho_i(L),\rho_i(c),\rho_i(U))=(1,2,3)\quad(i\ne j).
\]
\end{lemma}
\begin{proof}
Fixed rows agree at coordinate $j$, so none are strictly comparable. By the implication criterion, any failure therefore gives a sequence
\[
                 L\prec_*c_1\prec_*\cdots\prec_*c_t\prec_*U,
\]
where $L,U$ are fixed and each $c_i$ has one star. The first comparison forces the star of $c_1$ to occur at $j$, since otherwise both entries there would be $1$.

If $t=2$, a coordinate outside the two star positions gives four strictly increasing values in $\{1,2,3\}$, which is impossible. If $t\ge3$, at least one of $c_2,c_3$ is starred at $j$: otherwise their entries there are equal. Thus the stars of $c_1,c_2,c_3$ occupy at most two coordinates. Since $k\ge3$, a coordinate outside them gives four strictly increasing values for $L,c_1,c_2,c_3$, again impossible.

Hence $t=1$. Its star is at $j$, and the three entries at every other coordinate are necessarily $1,2,3$. Conversely, the displayed pattern forces the star of $c$ to be both $0$ and $4$.
\end{proof}

\section{Balanced parts: probability estimates}\label{sec:prob}
We now assume $|A|=|B|=2k$. By Lemma~\ref{lem:blocks}, it is enough to find an oriented matching whose rank table has a feasible completion. We first keep the matching fixed and randomise its orientations. When this estimate is insufficient, we randomise the matching as well. Only the final case $k=5$ needs the full freedom in the star assignments.

\subsection{Orienting a fixed matching: \texorpdfstring{$k\ge8$}{k at least 8}}
Fix any unoriented perfect matching of $A$ and independently orient its $k$ pairs. Assign every one-star row rank $0$, and assign every multi-star row both extremes as above.

For a given unordered pair of $B$ rows, the probability of strict comparability is at most $2^{1-k}$. Multi-star rows can never participate. Two one-star rows both contain a $0$, so neither can be strictly larger. For two fixed rows, each coordinate must have unequal ranks; reversing a pair reverses the inequality at that coordinate, so at most two of the $2^k$ orientations give a consistent strict direction. For a one-star row and a fixed row, the former must be the smaller; each of the other $k-1$ coordinates fixes an orientation, while the starred coordinate is unrestricted. Again there are at most two choices.

Consequently the probability of any strict comparison is at most
\begin{equation}\label{eq:orientation}
              \binom{2k}{2}2^{1-k}.
\end{equation}
At $k=8$ this is $120/128=15/16<1$. Its ratio at consecutive values is
\[
 \frac{(k+1)(2k+1)}{2k(2k-1)}<1\qquad(k\ge8),
\]
so it stays below one. Some orientation has no strict comparison. Lemma~\ref{lem:blocks} supplies the word.

\subsection{Averaging over matchings: \texorpdfstring{$k=6,7$}{k = 6, 7}}\label{sec:exactcount}
Choose uniformly among all oriented perfect matchings of $A$. The endpoints within each pair are ordered, but the pairs themselves are unordered. For each prescribed $k$-element set of first endpoints, a matching is a bijection from that set to its complement, giving $k!$ choices. Thus their total number is
\[
                            Q_k=\binom{2k}{k}k!=\frac{(2k)!}{k!}.
\]
Keep the same simple assignment rule: one-star rows receive $0$, multi-star rows receive both extremes.

For two rows $b,c$, partition $A$ into $X,Y,Z,T$, according as a vertex is adjacent to $b$ only, to $c$ only, to both, or to neither. Let $x=|X|$, $y=|Y|$, $z=|Z|$ and $t=|T|$, so $x+y+z+t=2k$. For integer arguments, define
\begin{equation}\label{eq:matchingcount}
 M_k(x,y,z)=
 \begin{cases}
 \displaystyle\frac{x!y!z!}{(k-x)!(k-y)!(k-z)!},
   &x+y+z=2k,\quad 0\le x,y,z\le k,\\[6pt]
 0,&\text{otherwise}.
 \end{cases}
\end{equation}
This counts matchings using only pairs between the three indicated classes, with each pair's orientation prescribed by its classes. Indeed, the numbers of $XY,XZ,YZ$ pairs must be
\[
                 \alpha=k-z,\qquad\beta=k-y,\qquad\gamma=k-x.
\]
Choose which vertices in each class go to each of the other two classes, and then choose the three bijections across those classes. The number of choices is
\[
 \frac{x!}{\alpha!\beta!}\,
 \frac{y!}{\alpha!\gamma!}\,
 \frac{z!}{\beta!\gamma!}\,
 \alpha!\beta!\gamma!
       =\frac{x!y!z!}{\alpha!\beta!\gamma!}.
\]
The stated conditions are exactly those ensuring that the three pair counts are nonnegative integers.

To have $r_i(b)<r_i(c)$ for every $i$, the possible directed pair types are $XY,XZ,ZY$, plus at most one of $YY,TY,YT$. The last three make the smaller row's sole zero. A $YY$ pair has two orientations of its distinct vertices. No other pair type works. Thus $t\ge2$ gives probability zero. Summing both possible directions of comparison, the numerator of the probability of strict comparability is
\begin{equation}\label{eq:Pzero}
 P_k(x,y,z,0)=2M_k(x,y,z)
 +y(y-1)M_{k-1}(x,y-2,z)
 +x(x-1)M_{k-1}(x-2,y,z),
\end{equation}
and
\begin{equation}\label{eq:Pone}
 P_k(x,y,z,1)=2yM_{k-1}(x,y-1,z)
             +2xM_{k-1}(x-1,y,z).
\end{equation}
Here a directed type such as $XY$ means that the first endpoint lies in $X$ and the second in $Y$. When $t=0$, choosing the exceptional $YY$ pair gives $y(y-1)$ oriented choices; the symmetric comparison gives the term involving $x(x-1)$. When $t=1$, its vertex must be paired with a vertex of $Y$ or $X$, and either orientation is possible. These account for the factors $2y$ and $2x$. Set $P_k(x,y,z,t)=0$ for $t\ge2$. Thus the probability of a strict comparison, in either direction, is exactly $P_k(x,y,z,t)/Q_k$.

The following table gives the maxima of the integer numerators over $x,y$ with $x+y=2k-t-z$, separately for each $z$. All omitted $z>k$ give zero. These are finite evaluations of factorials in \eqref{eq:matchingcount}--\eqref{eq:Pone}, checked entry by entry by the verification command in Appendix~\ref{app:verification}.
\begin{center}
\small
\begin{tabular}{ccrrrrrrrr}
\toprule
$k$&$t$&$z=0$&1&2&3&4&5&6&7\\
\midrule
6&0&5040&5040&7200&8640&8064&5040&1440&--\\
6&1&1440&2400&2880&3456&2880&1440&0&--\\
7&0&40320&40320&60480&79200&86400&72000&40320&10080\\
7&1&10080&17280&21600&28800&28800&23040&10080&0\\
\bottomrule
\end{tabular}
\end{center}
Therefore
\[
 \max\Pr(\text{strict comparison})=
 \begin{cases}
 8640/665280=1/77,&k=6,\\
 86400/17297280=5/1001,&k=7.
 \end{cases}
\]
The union bounds are respectively
\begin{equation}\label{eq:unions67}
                  \binom{12}{2}\frac1{77}=\frac67<1,
 \qquad \binom{14}{2}\frac5{1001}=\frac5{11}<1.
\end{equation}
Both cases follow.

\subsection{Five pairs}
Here it is better to allow the optimal star assignment and use Lemma~\ref{lem:cores}. The probability that two given rows form a strictly comparable fixed pair is $2M_5(x,y,z)/Q_5$. The maximum of $M_5$ is $216$, attained at a permutation of $(3,3,4)$, and $Q_5=30240$. Thus the probability of a fixed-pair obstruction is at most $1/70$.

For $k=5$, the pair counts $(\alpha,\beta,\gamma)=(5-z,5-y,5-x)$ satisfy $\alpha+\beta+\gamma=5$. Up to permutation, the possibilities and the corresponding values of $M_5(x,y,z)$ are
\begin{center}
\begin{tabular}{crrrrr}
\toprule
$(\alpha,\beta,\gamma)$ &$(0,0,5)$&$(0,1,4)$&$(0,2,3)$&$(1,1,3)$&$(1,2,2)$\\
$M_5(x,y,z)$&120&120&120&192&216\\
\bottomrule
\end{tabular}
\end{center}

The one-star row in a three-row obstruction of Lemma~\ref{lem:cores} has degree $2k-2$. Let $d$ be the number of $B$ vertices with that degree. The fixed rows have degree $k$ or $k+1$, since they have one neighbour in each of the other $k-1$ pairs and one or two in the special pair. For $k=5$ these degrees are different from $2k-2$. There are therefore at most
\begin{equation}\label{eq:triples5}
                         d\binom{10-d}{2}\le63
\end{equation}
choices of a central row and an unordered pair of fixed rows.

For each choice, the central row's two non-neighbours must be matched together, with probability $1/(2k-1)$. Conditional on this, outside that pair the fixed rows' neighbourhoods must be complementary $(k-1)$-sets. Every prescribed set of $k-1$ first endpoints on the remaining vertices admits exactly $(k-1)!$ matchings. Their first-endpoint set is therefore uniform among the $\binom{2k-2}{k-1}$ possibilities. It must equal one fixed row's neighbourhood or the other's. Hence the probability of this triple obstruction is at most
\begin{equation}\label{eq:tripleprob}
 \theta_k=\frac{2}{(2k-1)\binom{2k-2}{k-1}},\qquad \theta_5=\frac1{315}.
\end{equation}
This is an upper bound: it ignores further necessary conditions inside the central row's pair.

The probability that any obstruction occurs is consequently at most
\begin{equation}\label{eq:union5}
                  \binom{10}{2}\frac1{70}+63\frac1{315}
                  =\frac{59}{70}<1.
\end{equation}
Some matching is feasible, so Lemma~\ref{lem:blocks} proves the balanced ten-by-ten case.

\section{The finite cases}\label{sec:finite}
The probability estimates leave balanced parts of sizes six and eight. We prove the two remaining assertions by certificates of Boolean unsatisfiability. The common idea is to describe a family that obstructs every oriented matching, then show that such a family cannot exist.

\begin{lemma}\label{lem:three}
Every family of at most six distinct subsets of a six-element set has a feasible oriented matching.
\end{lemma}

\begin{lemma}\label{lem:four}
Every family of at most eight distinct subsets of an eight-element set has a feasible oriented matching.
\end{lemma}

A subset of $\{0,\ldots,2k-1\}$ is encoded by the integer whose binary expansion has ones precisely at its elements. We call this integer its \emph{mask}. The neighbourhoods therefore range over $64$ masks when $k=3$, and $256$ when $k=4$.

Simultaneous reversal of all pairs replaces each prescribed rank $\rho_i$ by $4-\rho_i$ and interchanges the two extreme completions. It preserves feasibility. We may thus fix the orientation of the pair containing vertex $0$. There are $15\cdot2^{3-1}=60$ representatives for three pairs and $105\cdot2^{4-1}=840$ for four pairs. Reordering the pairs only permutes coordinates and also preserves feasibility.

\subsection{Obstructions for four pairs}\label{sec:catalogue}
For four pairs, Lemma~\ref{lem:cores} gives the complete list of minimal obstructions. A comparable fixed pair has, in each coordinate, one of the rank pairs
\[
                          (1,2),\ (1,3),\ (2,3).
\]
There are exactly $3^4=81$ such unordered pair obstructions. Each has a unique smaller row, so these ordered rank patterns count each obstruction once.

In a three-row obstruction, the central row has a star at one coordinate and rank $2$ at the other three. The two fixed rows $L,U$ have ranks $1$ and $3$, respectively, at those other coordinates. At the special coordinate, let $\lambda,\mu\in\{1,2,3\}$ be their respective ranks. If $\lambda<\mu$, the fixed rows already form a pair obstruction. Otherwise no proper subfamily is infeasible. Thus the six possibilities $\lambda\ge\mu$ give $4\cdot6=24$ minimal triple obstructions.

Each representative matching therefore has $105$ minimal obstructions. Translating the rank patterns back into subsets of the eight vertices and taking their union gives $4,935$ distinct obstructions: $2,415$ pairs and $2,520$ triples. The supplied catalogue is generated directly from these patterns.

\subsection{Implication paths for three pairs}\label{sec:three-certificate}
For three pairs, an obstruction may involve a longer implication path. We describe these paths by the sets of rows needed to force a Boolean value. This retains the small rank space $\{\st,1,2,3\}^3$ and avoids searching over all six-element families of neighbourhoods.

For each one-star row $b$, a \emph{support} is a set of rows containing $b$ from which the implications force $t_b=1$. We construct a list $\mathcal W_b$ of such supports by the following rules, retaining only sets of size at most six:
\begin{enumerate}[label=(\roman*)]
\item If $u$ is fixed and $b\prec_*u$, include $\{b,u\}$ in $\mathcal W_b$.
\item If $S\in\mathcal W_b$, $c$ has one star, and $c\prec_*b$, include $S\cup\{c\}$ in $\mathcal W_c$ whenever its size is at most six.
\end{enumerate}
Add a support only if it contains no support already retained for the same row, and then discard its strict supersets from that row's list. Continue until neither rule adds a support. Each addition enlarges the collection of sets containing a retained support, and no discarded support can become inclusion-minimal again. Since the rank space is finite, the procedure terminates.

Each retained support forces its row to $1$: the first rule is direct forcing, and the second follows the implication $t_b=1\Rightarrow t_c=1$. Conversely, induction along an implication path shows that every deduction within a selected family of at most six rows contains a retained support for its final row. Discarding a larger support does not affect this conclusion, since any subsequent step from the smaller support remains contained in the corresponding step from the larger one.

To obtain obstructions, include every comparable fixed pair, and every set $S\cup\{l\}$ of size at most six for which $S\in\mathcal W_b$, $l$ is fixed and $l\prec_*b$. The latter set forces $t_b$ to both values. Retain the inclusion-minimal sets in this list. A selected family of at most six rows is infeasible if and only if it contains one of them: the forward implication follows from the preceding induction and the implication criterion; the reverse follows from the conflicting deductions.

For the canonical three-coordinate rank table, this construction gives $2,517$ pairs $(b,S)$ and $759$ obstructions. Their sizes and numbers are
\begin{center}
\begin{tabular}{crrrrr}
\toprule
Size &2&3&4&5&6\\
Number &27&18&6&204&504\\
\bottomrule
\end{tabular}
\end{center}
Transporting these obstructions through the $60$ representative matchings gives $38,430$ distinct subsets of the $64$ neighbourhood masks. For each representative, the associated list contains $759$ obstructions. The finite checks verify the two support rules, every terminal conflict, and all fixed-pair conflicts. These are precisely the local conditions used in the induction above.

\subsection{Boolean encoding and symmetry}\label{sec:cnf}
The obstruction lists reduce the two finite assertions to a selection problem: can at most $2k$ neighbourhoods contain an obstruction for every oriented matching? We express this question as a Boolean formula.

Let $k\in\{3,4\}$ and $h=2^{2k}$. For each mask $v\in\{0,\ldots,h-1\}$ introduce a Boolean variable $s_v$, with $s_v=1$ meaning that $v$ belongs to the selected family. For each obstruction $C$ in the union of the matching lists, introduce a variable $c_C$ and require
\begin{equation}\label{eq:implication}
                c_C\Longrightarrow s_v\qquad(v\in C).
\end{equation}
For every representative matching $T$, require
\begin{equation}\label{eq:cover}
                       \bigvee_{C\text{ an obstruction of }T}c_C.
\end{equation}
Finally impose $\sum_{v=0}^{h-1}s_v\le2k$. Thus each matching must have an obstruction whose neighbourhoods all belong to the selected family.

We also use the symmetry under relabelling the $2k$ vertices. A permutation $p$ acts on masks by
\begin{equation}\label{eq:mask-action}
 p_*(v)=\sum_{a=0}^{2k-1}\varepsilon_a(v)2^{p(a)},
 \qquad \varepsilon_a(v)\in\{0,1\},
\end{equation}
where $\varepsilon_a(v)$ is the $a$th binary digit of $v$. Reading masks from $h-1$ down to $0$, require $(s_v)_v$ to be lexicographically at least $(s_{p_*(v)})_v$ for each adjacent transposition and each nontrivial cyclic rotation. These give $10$ comparisons when $k=3$ and $14$ when $k=4$. Every family can be relabelled so that its selector vector is largest in its full permutation orbit. The additional comparisons therefore preserve the existence of a family obstructing every matching.

Let $F_k$ be the Boolean formula specified in Appendix~\ref{app:clauses}, which expresses these requirements as a conjunction of clauses using auxiliary variables for the cardinality bound and the lexicographic comparisons.

\begin{proposition}[From an obstructing family to a formula]\label{prop:encoding}
For $k=3,4$, the formula $F_k$ is satisfiable if and only if there is a family of at most $2k$ neighbourhoods on a $2k$-element set with no feasible oriented matching.
\end{proposition}
\begin{proof}
Suppose such a family exists. Relabel the vertices so that its selector vector is largest in its permutation orbit. For every obstruction $C$, assign $c_C=1$ exactly when $C$ is contained in the family. Each matching fails, so its list contains an obstruction and~\eqref{eq:cover} holds. The membership implications~\eqref{eq:implication} also hold. The family has at most $2k$ members and satisfies every lexicographic comparison, so the auxiliary variables can be assigned as in Appendix~\ref{app:clauses}. The additional initial clause specified there holds because~\eqref{eq:cover} forces a positive obstruction variable. Hence $F_k$ is satisfiable.

Conversely, a satisfying assignment selects at most $2k$ neighbourhoods. For each representative matching,~\eqref{eq:cover} supplies a true obstruction variable, and~\eqref{eq:implication} places all its rows in the selected family. Hence every representative fails. Coordinate permutations and simultaneous reversal cover all oriented matchings, so none is feasible.
\end{proof}

\subsection{Checking the contradictions}\label{sec:rup}
Both formulas are unsatisfiable. Each certificate consists of clause additions verified by one elementary rule: if assigning every literal of a proposed clause false leads to a contradiction by unit propagation, the current formula entails that clause. Unit propagation uses a clause whose other literals are false to force its remaining literal. Adding an entailed clause preserves satisfiability. A sequence ending in the empty clause therefore proves unsatisfiability. This is the reverse-unit-propagation (RUP) rule; Appendix~\ref{app:verification} gives its soundness argument and a small example.

For $F_4$, the certificate has $34,482$ additions, including its final empty clause. For $F_3$, it suffices to retain the $141,926$ original clauses used in the derivation of a contradiction. The supplied refutation of this subformula has $3,305$ additions. Every retained clause is identified in $F_3$, with its original variable numbering unchanged. A satisfying assignment of $F_3$ would satisfy this subformula, so its refutation proves the required contradiction.

The verification command checks the support and obstruction data, regenerates the mask catalogues and formulas, and verifies their agreement with the supplied files. It then checks the clause additions. The three-pair certificate includes lists of clauses specifying each unit-propagation derivation. These lists shorten verification; the checker verifies each indicated inference. The completeness of the obstruction lists and both refutations also have Lean proofs, together with the implications from obstructing families to satisfying assignments of the refuted formulas. Reproduction instructions are in Appendix~\ref{app:verification}.

Proposition~\ref{prop:encoding} now proves Lemmas~\ref{lem:three} and~\ref{lem:four}. A graph with repeated neighbourhoods has no more distinct masks than vertices in its second part. After choosing a feasible matching for those masks, restore the repetitions as explained after Lemma~\ref{lem:blocks}. That lemma gives the balanced bound for $k=3,4$. Combining these cases with Section~\ref{sec:prob} proves
\begin{equation}\label{eq:balanced}
        |A|=|B|=2k,\quad k\ge3\quad\Longrightarrow\quad\RN(G)\le k.
\end{equation}

\section{All remaining partite sizes}\label{sec:all}
The same construction yields the bound for the remaining part sizes. Let $m=|A|\le|B|=n$.

\subsection{Adding one permutation}
The construction of Mozhui and Krishna~\cite[Section~2]{MK} gives the following bound by appending one permutation.

\begin{lemma}\label{lem:baseline}
If $|A|=2k$ with $k\ge2$, then every bipartite graph on $A\sqcup B$ is $(k+1)$-representable.
\end{lemma}
\begin{proof}
Use the blocks and intervening words \eqref{eq:block}--\eqref{eq:word}, but choose arbitrary permutations of $B$ that are nondecreasing in their respective rank coordinates and assign every star rank $0$. The completed rows need not satisfy the hypothesis of Lemma~\ref{lem:orders}. Append the permutation
\begin{equation}\label{eq:extra}
               y_1x_1\,\rev(\pi_1)\,y_2x_2\cdots y_kx_k.
\end{equation}
Every letter now occurs $k+1$ times. Every nonedge between different $A$ pairs, or between $A$ and $B$, already has equal consecutive letters in the base restriction; appending letters cannot remove this obstruction. The same-pair order $y_i x_i$ in \eqref{eq:extra} breaks the order $x_i y_i x_i y_i$ of its own block, including when $k=2$ and the intervening words are empty. A $B$ pair that has not already failed alternation has constant order across the base block permutations, and reversing $\pi_1$ breaks that order.

For an edge incident with pair $1$, its alternating base restriction begins with the $A$ letter; the $A$ letter also precedes the $B$ letter in~\eqref{eq:extra}, so alternation continues. For an edge incident with any other pair, the base restriction begins with $B$, because $W_1$ supplies its first occurrence before the first copy of that $A$ letter. The $B$ letter also precedes the $A$ letter in~\eqref{eq:extra} for these pairs. Hence all edges keep alternating.
\end{proof}

Padding the smaller part with isolated vertices gives $\RN(G)\le1+\ceil{m/2}$ whenever $m\ge3$. More generally, padding any $m\le4$ to four gives $\RN(G)\le3$, enough for $m+n\ge9$.

If $m=2k\ge6$ and $n>m$, then $\ceil{(m+n)/4}\ge k+1$, so Lemma~\ref{lem:baseline} suffices. If $n=m$, use \eqref{eq:balanced}.

If $m=2k-1\ge5$ and $n\ge m+3$, then $\ceil{(m+n)/4}\ge k+1$, and padding to $2k$ followed by Lemma~\ref{lem:baseline} again suffices. It remains to consider an odd smaller part and a larger part whose size differs from it by at most two.

\subsection{Odd parts of nearly equal size}
Mozhui and Krishna~\cite[Theorem~3.1]{MK} proved that $\RN(G)\le\ceil{m/2}$ whenever the part sizes satisfy $5\le m\le n$ and $m$ is odd. For the remaining part sizes, the distinguished-coordinate lemma gives a short proof of this bound, including when $m=5$. The choice of one vertex will exclude every possible obstruction; all other pairs and their orientations can then be chosen arbitrarily.

\begin{lemma}\label{lem:odd-near-balanced}
Let $k\ge3$. Every bipartite graph with parts $A,B$ satisfying $|A|=2k-1$ and $|B|\le2k+1$ has a $k$-uniform representation.
\end{lemma}
\begin{proof}
We first choose $a\in A$. If the neighbourhood $A\setminus\{a\}$ is absent from $B$ for some $a$, choose that vertex. Otherwise select one vertex of $B$ for each of the $2k-1$ distinct neighbourhoods $A\setminus\{a\}$. These selected vertices have degree $2k-2\ne k$. At most two other vertices of $B$ remain, counting multiplicity. Hence at most two vertices have degree $k$. If there is at most one, choose any $a$. If there are two, choose $a$ outside the neighbourhood of one of them; this is possible since $k<2k-1$. In every case, our choice ensures that either $A\setminus\{a\}$ is absent or no two distinct degree-$k$ vertices of $B$ are both adjacent to $a$.

Adjoin an isolated vertex $z$ to $A$ and make $(a,z)$ the first oriented pair. Pair and orient the remaining vertices arbitrarily. Every row has rank $1$ or a star in the first coordinate. By Lemma~\ref{lem:special-coordinate}, infeasibility would require a central vertex $c$ with neighbourhood $A\setminus\{a\}$ and two distinct fixed vertices $L,U$. Each fixed vertex would be adjacent to $a$ and exactly one endpoint of every other pair, so both would have degree $k$. The choice of $a$ excludes this possibility.

The matching is therefore feasible. Lemma~\ref{lem:blocks} gives a $k$-uniform representation of the enlarged graph, and deleting $z$ gives the required representation.
\end{proof}

For the remaining part sizes $m=2k-1\ge5$ and $m\le n\le m+2$, we have $\lceil(m+n)/4\rceil=k$. Lemma~\ref{lem:odd-near-balanced} therefore gives the desired bound. Together with the even cases and the bound for $m\le4$, this covers every bipartition with $m+n\ge9$ and proves Theorem~\ref{thm:main}.\hfill$\square$

\paragraph{\textbf{Acknowledgments.}}
The authors thank the Isaac Newton Institute for Mathematical Sciences, Cambridge, for support and hospitality during the programme \emph{Geometric spectral theory and applications}, where this work was undertaken. This work was supported by EPSRC grant EP/Z000580/1.

\paragraph{\textbf{AI declaration.}}

The underlying construction using paired vertices comes from Mozhui and Krishna. Our principal advance is to choose the orders within its blocks so that the additional permutation becomes unnecessary in the previously unresolved case of two parts of equal even size. This approach arose through discussions between the authors and ChatGPT 5.5 and 5.6. For at least four pairs, a key structural insight is that a long contradictory chain would require four strictly increasing ranks at a coordinate where only $1,2,3$ are available. This yields the obstruction classification used in the probability estimates and the four-pair certificate. The classification also emerged from discussions between the authors and ChatGPT 5.6, after the authors directed attention to the implications between the possible assignments of the unspecified ranks. Building on this classification, the authors directed further discussions with ChatGPT 6 towards the remaining odd part sizes. These discussions led to the distinguished-coordinate argument: adjoining an isolated vertex restricts the possible obstructions, and a suitable choice of its partner excludes those that remain. We used ChatGPT 5.6 through Codex to assist with combinatorial counting and with the construction and verification of the finite computational certificates. After developing the manuscript, we used ChatGPT 6 through Codex to assist with a Lean formalisation of the proof. The authors reviewed the formal statements against the definitions, hypotheses and conclusions in the manuscript. Lean Comparator was used to check the formal proofs against separately reviewed theorem statements, and the axiom dependencies of the final theorems were inspected to exclude proof gaps, additional unproved axioms and native-execution trust. The formalisation also informed subsequent revisions to the mathematical arguments and their exposition. All content in the paper was reviewed and adopted by the authors, who take full responsibility for it.

\appendix
\section{The Boolean encoding}\label{app:clauses}
This appendix specifies the formula $F_k$ used in Section~\ref{sec:cnf}. A literal is a Boolean variable or its negation. A clause is a disjunction of literals, and a formula in conjunctive normal form (CNF) is a conjunction of clauses. We use $1$ and $0$ for the Boolean constants true and false.

The membership implications~\eqref{eq:implication} give the clauses
\[
                         \neg c_C\lor s_v\qquad(v\in C),
\]
and~\eqref{eq:cover} is already a clause for each representative matching. The remaining clauses express the cardinality bound and the lexicographic comparisons.

\subsection{Counting the selected neighbourhoods}
Let $k\in\{3,4\}$ and $h=2^{2k}$. Introduce variables $q_{i,j}$ for $1\le i\le h$ and $1\le j\le\min\{i,2k+1\}$. The condition $q_{i,j}=1$ means that at least $j$ of $s_0,\ldots,s_{i-1}$ are true. With boundary values $q_{i,0}=1$ for $0\le i\le h$ and $q_{i,j}=0$ for $i<j\le2k+1$, the recurrence is
\begin{equation}\label{eq:counter}
 q_{i,j}\ \longleftrightarrow\
 \bigl(q_{i-1,j}\lor(s_{i-1}\land q_{i-1,j-1})\bigr).
\end{equation}
The final clause is $\neg q_{h,2k+1}$. Thus thresholds through $7$ or $9$ suffice to exclude a seventh or ninth selected neighbourhood, respectively.

For Boolean literals or constants $z,u,x,w$, the equivalence $z\leftrightarrow\bigl(u\lor(x\land w)\bigr)$ is expressed by the four clauses
\begin{equation}\label{eq:counterclauses}
 (\neg u\lor z),\quad
 (\neg x\lor\neg w\lor z),\quad
 (\neg z\lor u\lor x),\quad
 (\neg z\lor u\lor w).
\end{equation}
The first two give the implication from right to left; the last two give the converse, since $(u\lor x)\land(u\lor w)=u\lor(x\land w)$. Substituting
\[
 z=q_{i,j},\quad u=q_{i-1,j},\quad x=s_{i-1},\quad w=q_{i-1,j-1}
\]
gives~\eqref{eq:counter}. Induction on $i$ proves the stated meaning of every counter variable. The clause $\neg q_{h,2k+1}$ therefore imposes exactly $\sum_{v=0}^{h-1}s_v\le2k$.

The implementation represents false by one variable $f$ together with the clause $\neg f$, and true by $\neg f$.

\subsection{Lexicographic comparisons}
Fix one of the permutations used in Section~\ref{sec:cnf}. Read the selector bits in decreasing order of their masks. The literal $e$ is true exactly when all earlier bit pairs agree; initially it is true. At a nonfixed position, let $x$ be the current selector and $y$ its permuted selector. The clause
\[
                         \neg e\lor x\lor\neg y
\]
excludes a first difference of the form $(x,y)=(0,1)$. Introduce a new variable $e'$ for equality of the extended prefix, with
\begin{equation}\label{eq:prefix}
                        e'\longleftrightarrow\bigl(e\land(x\leftrightarrow y)\bigr).
\end{equation}
This equivalence is encoded by
\begin{align*}
 &\neg e'\lor e, &&\neg e'\lor\neg x\lor y,
       &&\neg e'\lor x\lor\neg y,\\
 &\neg e\lor\neg x\lor\neg y\lor e',
       &&\neg e\lor x\lor y\lor e'.
\end{align*}
The first three clauses prove the forward implication. For the converse, if $e$ is true and $x=y$, one of the last two clauses forces $e'$; otherwise $e'$ must be false. Fixed positions need no clause or auxiliary variable.

The largest selector vector in each full permutation orbit satisfies every comparison. Relabelling preserves the cardinality of a family and the property of obstructing every matching, so these clauses preserve the existence of such a family.

\subsection{The complete formulas}\label{app:formula-dimensions}
The formula $F_k$ consists of the membership clauses, the matching clauses~\eqref{eq:cover}, the counter clauses, the symmetry clauses and the clause $\neg f$. The implementation also includes an initial clause containing every variable positively. This clause is redundant, since~\eqref{eq:cover} forces a positive obstruction variable.

The full formulas have the following dimensions. The last column counts variables associated with obstructions.
\begin{center}
\small
\begin{tabular}{crrrrrrr}
\toprule
$k$ & Clauses & Variables & Selectors & Constant & Counters & Prefixes & Obstructions\\
\midrule
3&222,795&39,382&64&1&427&460&38,430\\
4&38,240&10,116&256&1&2,268&2,656&4,935\\
\bottomrule
\end{tabular}
\end{center}
In each row, the five variable classes sum to the total. The larger three-pair formula reflects the longer implication paths that can occur in three coordinates.

\section{Verification and reproduction}\label{app:verification}
The accompanying repository~\cite{CDrepo} contains the Lean proofs in \texttt{lean/} and the finite data and verification programs in \texttt{validation/}. All file paths below are relative to the repository root. The program \texttt{validation/verify.py} checks the two certificate arguments of Section~\ref{sec:finite} and the exact arithmetic of Section~\ref{sec:prob}.

\subsection{From rank patterns to clauses}
For four pairs, the generator lists the $81$ pair and $24$ triple patterns for each representative matching. For three pairs, the support construction in Section~\ref{sec:three-certificate} supplies the obstruction lists. Verification checks direct forcing, closure under each implication, terminal conflicts and fixed-pair conflicts, together with the deductions represented by the supports. Each rank row determines a unique neighbourhood mask for a fixed oriented matching, so these checks transfer to the lists of mask subsets. In the initial-rank data, the integer $0$ denotes $\st$; it is distinct from the completed rank $0$ used for the first insertion gap.

The formulas use one-based DIMACS variable identifiers. A positive integer denotes a variable, its negative denotes its negation, and $0$ terminates a clause. The variables are allocated in the order shown in the table in Appendix~\ref{app:formula-dimensions}: selectors, the false constant, counters, prefix equalities and obstructions. Thus selectors have identifiers $1,\ldots,h$, and the false variable has identifier $h+1$. A negative unit clause forces that variable to be false; its negative literal represents true.

For the symmetry clauses, the vertex permutations are $p=(a\ a+1)$ for $0\le a\le2k-2$ and $p(a)=a+r\pmod{2k}$ for $1\le r\le2k-1$. The compared bits are $s_v$ and $s_{p_*(v)}$, with $p_*$ as in~\eqref{eq:mask-action}. Tautological clauses are omitted and repeated literals within a clause are removed. These operations preserve its truth value. Regeneration checks the resulting clause lists against the exact formulas supplied to the certificate checkers.

The three-pair refutation uses a subformula of $F_3$. The file \path{validation/data/pairing6-provenance.json} identifies each retained clause in the full formula. Clause identifiers in the propagation lists are renumbered to refer to this subformula and the preceding additions; variable identifiers are unchanged. Checking this correspondence ensures that the contradiction applies to the formula constructed from neighbourhoods. The subformula and its refutation are \path{validation/data/pairing6.cnf} and \path{validation/data/pairing6.lrat}; the support and obstruction data are in \path{validation/data/three-pair-closure.json}. For four pairs, the formula and refutation are \path{validation/data/pairing8.cnf} and \path{validation/data/pairing8.rup.gz}.

\subsection{Soundness of certificate verification}
The verification uses the unit-propagation principle of Goldberg and Novikov~\cite{GN}. Suppose $F$ is the conjunction of the original formula and the additions already checked, and $C=\ell_1\lor\cdots\lor\ell_t$ is a proposed next clause. Assigning every $\ell_i$ false amounts to assuming $\neg C$. If unit propagation then gives a contradiction, $F\land\neg C$ is unsatisfiable, so $F$ entails $C$. Induction shows that every checked addition follows from the original formula. The final empty clause proves unsatisfiability.

For example, consider
\[
 (p\lor q)\land(p\lor\neg q)\land(\neg p\lor q)\land(\neg p\lor\neg q).
\]
The clause $p$ passes the test: assuming $\neg p$, the first two clauses force both $q$ and $\neg q$. Once $p$ has been added, the last two clauses yield a contradiction. The certificates use this same rule. The three-pair certificate additionally names the clauses to use in each propagation sequence; the checker verifies that each is unit under the assignments already derived, or gives the terminal contradiction.

The certificates were obtained with Z3~\cite{Z3}, version 4.13.3. Their verification requires only the formulas and the clause additions. The four-pair certificate is checked by two implementations, one using watched literals and the other using occurrence counts. For three pairs, the supplied propagation lists permit a direct check of each inference. The Lean development separately checks propositional proofs of the contradictions and the implications from obstructing families to satisfying assignments.

\subsection{Reproduction}
With Python 3.10 or later and a GCC-compatible C++17 compiler available, run from the repository root:
\begin{Verbatim}[fontsize=\small]
python3 validation/verify.py
\end{Verbatim}
The command checks the support and obstruction data, regenerates both formulas, checks the three-pair subformula correspondence, verifies both refutations, and checks every entry of the probability tables together with the rational bounds. No solver or internet connection is needed for these checks. The option \texttt{--cxx} specifies an alternative compiler executable. The file \path{validation/README.md} describes the individual checks, and \path{validation/verification.json} contains their results.

\subsection{Lean formalisation}
The file \path{lean/lean-toolchain} fixes the Lean version at 4.33.1, and \path{lean/lakefile.toml} and \path{lean/lake-manifest.json} fix the library versions. The file \path{lean/Crown/Main.lean} contains Theorem~\ref{thm:main}, Corollary~\ref{cor:crown} and the corresponding statement for uniform word length. The exact crown representation numbers are proved in \path{lean/Crown/CrownValues.lean}. The formal proofs of the finite graph assertions are in \path{lean/Crown/CertificateSixGraph.lean} and \path{lean/Crown/CertificateGraph.lean}; \path{lean/Crown/OddGraphs.lean} treats the odd part sizes, including the five-by-seven case. The file \path{docs/mathematics.md} gives the correspondence between the remaining arguments and their Lean modules.

After installing Lean through Elan, the complete development can be compiled from the repository root by
\begin{Verbatim}[fontsize=\small]
cd lean
lake exe cache get
lake build Solution
\end{Verbatim}
The repository README gives the installation requirements. The file \path{lean/Solution.lean} assembles the main results and supporting lemmas. Lean Comparator checks all statements listed in \path{lean/full-comparator.json} against the separate specifications assembled in \path{lean/Challenge.lean}, and the exported proofs are checked by the Lean kernel. The proofs use only propositional extensionality, classical choice and quotient soundness. The file \path{tools/README.md} gives the commands for reproducing the comparison and axiom checks.


\small
\begin{thebibliography}{99}
\raggedright
\bibitem{AKGZ} \"O. Akg\"un, I. Gent, S. Kitaev and H. Zantema,
\emph{Solving computational problems in the theory of word-representable graphs},
J. Integer Seq. 22 (2019), Article 19.2.5.
\href{https://cs.uwaterloo.ca/journals/JIS/VOL22/Kitaev/kitaev11.html}{Journal version}.

\bibitem{CDrepo} M. J. Colbrook and C. Drysdale,
\emph{Crown graphs maximise the representation number of bipartite graphs: formal proofs and certificates},
software and data, 2026.
\url{https://github.com/miinadietrich/crown_graphs_maximise}.

\bibitem{GKP} M. Glen, S. Kitaev and A. Pyatkin,
\emph{On the representation number of a crown graph},
Discrete Appl. Math. 244 (2018), 89--93.
\href{https://doi.org/10.1016/j.dam.2018.03.013}{doi:10.1016/j.dam.2018.03.013}.

\bibitem{GN} E. Goldberg and Y. Novikov,
\emph{Verification of proofs of unsatisfiability for CNF formulas},
in \emph{Proceedings of the Design, Automation and Test in Europe Conference and Exhibition},
IEEE Computer Society, 2003, pp.~886--891.
\href{https://doi.org/10.1109/DATE.2003.1253718}{doi:10.1109/DATE.2003.1253718}.

\bibitem{HHMO24} Z. Hefty, P. Horn, C. Muir and A. Owens,
\emph{Word-representable graphs: orientations, posets, and bounds},
Electron. J. Combin. 31 (2024), no.~4, Paper P4.2.
\href{https://doi.org/10.37236/12806}{doi:10.37236/12806}.

\bibitem{HHMO26} Z. Hefty, P. Horn, C. Muir and A. Owens,
\emph{Word-representation numbers of graphs: bottlenecks and bounds},
J. Combin. Theory Ser. A 223 (2026), Article 106215.
\href{https://doi.org/10.1016/j.jcta.2026.106215}{doi:10.1016/j.jcta.2026.106215}.

\bibitem{KL} S. Kitaev and V. Lozin,
\emph{Words and Graphs}, Monographs in Theoretical Computer Science.
An EATCS Series, Springer, Cham, 2015.
\href{https://doi.org/10.1007/978-3-319-25859-1}{doi:10.1007/978-3-319-25859-1}.

\bibitem{Z3} L. de Moura and N. Bj{\o}rner,
\emph{Z3: an efficient SMT solver},
in \emph{Tools and Algorithms for the Construction and Analysis of Systems},
Lecture Notes in Computer Science 4963, Springer, Berlin, 2008, pp.~337--340.
\href{https://doi.org/10.1007/978-3-540-78800-3_24}{doi:10.1007/978-3-540-78800-3\_24}.

\bibitem{MK} K. Mozhui and K. V. Krishna,
\emph{On the conjecture of the representation number of bipartite graphs},
preprint, 2025. \href{https://arxiv.org/abs/2506.01057}{arXiv:2506.01057v1}.

\bibitem{MKperm} K. Mozhui and K. V. Krishna,
\emph{An upper bound for the permutation-representation number of bipartite graphs},
J. Inf. Process. 33 (2025), 1033--1041.
\href{https://doi.org/10.2197/ipsjjip.33.1033}{doi:10.2197/ipsjjip.33.1033}.

\end{thebibliography}
\end{document}